\documentclass{amsart}

\usepackage{amsfonts,amssymb,amscd,amsmath,enumerate,verbatim,calc}

\newcommand{\sep}{\operatorname{sep}}
\newcommand{\Eig}{\operatorname{Eig}}

\newcommand{\reg}{\operatorname{reg}}

\newcommand{\bC}{{\boldsymbol C}}

\theoremstyle{plain}
\newtheorem{theorem}{Theorem}

\newtheorem{lemma}[theorem]{Lemma}

\theoremstyle{definition}

\theoremstyle{remark}
\newtheorem{remark}[theorem]{Remark}

\newcounter{hours}\newcounter{minutes}
\newcommand{\printtime}{%
        \setcounter{hours}{\time/60}%
        \setcounter{minutes}{\time-\value{hours}*60}%
        \thehours\,h\ \theminutes\,min}

\begin{document}

\title[Invariants of the dihedral group $D_{2p}$ in characteristic two]
{Invariants of the dihedral group $D_{2p}$ in characteristic two }
\date{\today,\ \printtime}
\author{ Martin Kohls}
\address{Technische Universit\"at M\"unchen \\
 Zentrum Mathematik-M11\\
Boltzmannstrasse 3\\
 85748 Garching, Germany}
\email{kohls@ma.tum.de}

\author{M\"uf\.it Sezer}
\address { Department of Mathematics, Bilkent University,
 Ankara 06800 Turkey}
\email{sezer@fen.bilkent.edu.tr}
%\email{mufit.sezer@boun.edu.tr}
\thanks{We thank T\" {u}bitak for
funding a visit of the first author to Bilkent University. Second
author is also partially supported by T\"{u}bitak-Tbag/109T384 and
T\"{u}ba-Gebip/2010. }

\begin{abstract}
We consider finite dimensional representations of the dihedral
group $D_{2p}$ over an algebraically closed field of
characteristic two where $p$ is an odd integer and study the
degrees of generating and separating polynomials   in the
corresponding ring of invariants. We give an upper bound for the
degrees of the polynomials in a minimal generating set   that does
not depend on $p$ when the dimension of the representation is
sufficiently large. We also show that $p+1$ is the minimal number such that
the invariants up to that degree always form a separating set. As well,
we give an explicit description of a separating set when $p$ is
prime.
\end{abstract}

 \maketitle

\section{introduction}
Let $V$ be a finite dimensional representation of a group $G$ over
an algebraically closed field $F$. There is an induced action of
$G$ on the algebra of polynomial functions $F[V]$ on $V$ that is
given by $g (f)=f\circ g^{-1}$ for $g \in G$ and $f\in F[V]$. Let
$F[V]^G$ denote the ring of invariant polynomials in $F[V]$. One
of the main goals in invariant theory  is to determine $F[V]^G$ by
computing the generators and the relations. One  may  also study
subsets in $F[V]^G$ that   separate the orbits just as well  as
the full invariant ring. A set $A\subseteq F[V]^G$ is said to be
separating for $V$ if for any pair of vectors $u,w\in V$, we have:
If $f(u)=f(w)$ for all $f\in A$, then $f(u)=f(w)$ for all $f\in
F[V]^G$. There has been a particular rise of interest in
separating invariants following the text book
\cite{{DerksenKemper}}. Over the last decade there has been an
accumulation of evidence that demonstrates that separating sets
are better behaved  and enjoy many properties that make them
easier to obtain. For instance, explicit separating sets are given
for all modular representations of cyclic groups of prime order in
\cite{MR14}. Meanwhile generating sets are known only for very
limited cases for the invariants of these representations. In
addition to attracting attention in their own right separating
invariants can be also used as a stepping stone to build up
generating invariants, see \cite{MR2388087}. For more background
and motivation on separating invariants we direct the reader to
\cite{{DerksenKemper}} and \cite{MR0000001}.

   In this paper we study the invariants of the dihedral
group $D_{2p}$ over a field of characteristic two where $p$ is an
odd integer. The invariants of dihedral groups in characteristic
zero have been worked out by Schmid in \cite{MR1180987} where she
sharpened Noether's bound for non-cyclic groups. Specifically,
among other things,  she proved that the invariant ring $\bC
[V]^{D_{2p}}$ is generated by polynomials of degree at most $p+1$.
Obtaining explicit generators or even sharp degree bounds is much
more difficult when the order of the group is divisible by the
characteristic of the field. The main difficulty is that the
degrees of the generators grow unboundedly as the dimension of the
representation increases. Recently, Symonds \cite{sym} established
that $F[V]^G$ is generated by invariants of degree at most $(\dim
V)(|G|-1)$ for any representation $V$ of any group $G$. In section
1 we improve Symonds' bound considerably for  $D_{2p}$ in
characteristic two. The bound we obtain is about half of $\dim
(V)$ and it does not depend on $p$ if the dimension of the
 part of $V$ where $D_{2p}$ does not act like its factor group ${\mathbf Z}/2{\mathbf Z}$ is large enough. In section 2 we turn our
attention to separating invariants for these representations. The
maximal degree of an element  in the  generating set for the
regular representation provides an upper bound for  the degrees of
separating invariants. We build on this fact and our results in
section 1 to compute the supremum of the degrees of polynomials in
(degreewise minimal) separating sets over all representations.
This resolves a conjecture in \cite{KohlsKraft} positively. Then
we specialize to the case where $p$ is a prime integer and
describe an explicit separating set for all representations of
$D_{2p}$. Our description is recursive and inductively yields a
set that is "nice" in terms of constructive complexity. The  set
consists of invariants that are in the image of the relative
transfer with respect to the subgroup of order $p$ of $D_{2p}$
together with the products of the variables over  certain
summands. Moreover, these polynomials depend on variables from at
most three summands.

\section{Notation and Conventions}\label{SetupSection} In this section we fix the
notation for the rest of the paper. Let $p\ge 3$ be an odd number
and $G:=D_{2p}$ be the dihedral group of order $2p$. We fix
elements $\rho$ and $\sigma$ of order $p$ and $2$ respectively.
Let $H$ denote the subgroup of order $p$ in $G$.  Let $F$ be an
algebraically closed field of characteristic two, and $\lambda\in
F$  a primitive $p$-th root of unity.

\begin{lemma}
For $0\le i\le (p-1)/2$ let $W_i$ denote the two
dimensional module spanned by the vectors $v_1$ and $v_2$ such
that $\rho (v_1)=\lambda^{-i} v_1$, $\rho (v_2)=\lambda^{i}v_2$,
$\sigma (v_1)=v_2$ and $\sigma (v_2)=v_1$. Then the $W_{i}$ together with the
trivial module represent a complete list of indecomposable $D_{2p}$-modules.
\end{lemma}

\begin{proof}
Let $V$ be any $D_{2p}$-module. As $p$ is odd, the action of $\rho$ is
diagonalizable. For any $k\in{\mathbb Z}$, $\sigma$ induces an
isomorphism of the eigenspaces of $\rho$, $\sigma:
\Eig(\rho,\lambda^{k})\stackrel{\sim}{\rightarrow}
\Eig(\rho,\lambda^{-k})$. Therefore as $D_{2p}$-module, $V$
decomposes into a direct sum of $\Eig(\rho,1)$ and some $W_{i}$'s
with $1\le i\le (p-1)/2$. The action of $\sigma$ on $\Eig(\rho,1)$
decomposes into a direct sum of trivial summands and summands
isomorphic to $W_{0}$.
\end{proof}

Note that $W_{i}$ is faithful if and only if $i$ and $p$ are coprime.
 Let $V$ be a reduced $G$-module, i.e., it does not contain the trivial module as a
summand. Assume that
$$V=\bigoplus_{i=1}^{r}W_{m_i}\oplus\bigoplus_{i=1}^sW_0,$$ where $r, s,
m_i$ are integers such that $r,s\ge 0$ and $0<m_i\le (p-1)/2$ for
$1\le i \le r$.
 By a suitable choice of basis we identify $V=F^{2r+2s}$ with a space of
 $2(r+s)$-tuples
 $\{(a_1, \dots , a_r, b_1, \dots , b_r, c_1,\dots ,c_s, d_1, \dots ,d_s) \mid a_i, b_i, c_j, d_j\in F, \, \; 1\le i\le r, 1\le j\le s \}$
such that the projection $(a_1, \dots , a_r, b_1, \dots , b_r,
c_1,\dots ,c_s, d_1, \dots ,d_s)\rightarrow (a_i, b_i)\in F^2$ is
a $D_{2p}$-equivariant surjection from $V$ to $W_{m_i}$ for $1\le
i\le r$ and the projection $(a_1, \dots , a_r, b_1, \dots , b_r,
c_1,\dots ,c_s, d_1, \dots ,d_s)\rightarrow (c_j, d_j)\in F^2$ is
a $D_{2p}$-equivariant surjection from $V$ to $W_0$ for $1\le j\le
s$. Let $x_1, \dots ,x_r, y_1, \dots ,y_r, z_1, \dots ,z_s, w_1,
\dots ,w_s$ denote the corresponding basis elements in $V^*$, so
we have
$$F[V]=F[x_1,\dots ,x_r, y_1, \dots , y_r, z_1, \dots , z_s, w_1,\dots , w_s],$$
with $\sigma$ interchanging $x_i$ with $y_i$ for $1\le i\le r$ and
$z_j$ with $w_j$ for $1\le j\le s$. The action of $\rho$ is
trivial on $z_j$ and $w_j$ for $1\le j\le s$. Meanwhile $\rho
(x_i)=\lambda ^{m_i}x_i$ and $\rho (y_i)=\lambda ^{-m_i}y_i$ for
$1\le i\le r$.

\section{ Generating invariants}
In this section we give an upper bound for the degree of generators for
$F[V]^G$. Here, $p\ge 3$ is an arbitrary odd number. We continue with the
introduced notation. For $1\le i\le r$ and $1\le j\le s$, let $a_i, b_i, c_j, d_j $ denote non-negative
 integers.
Let $m=x_1^{a_1}\dots x_r^{a_r}y_1^{b_1}\dots y_r^{b_r}z_1^{c_1}\dots z_s^{c_s}w_1^{d_1}\dots w_s^{d_s}$ be
 a monomial in $F[V]$. Since    $\rho$ acts on  a monomial by multiplication with a scalar, all monomials that
 appear in a polynomial in $F[V]^G$ are invariant under the action of $\rho$.
 For a monomial $m$ that is invariant under the action of $\rho$, we let
 $o(m)$ denote its orbit sum, i.e. $o(m)=m$ if $m\in F[V]^{G}$ and
 $o(m)=m+\sigma (m)$ if $m\in F[V]^{\rho}\setminus F[V]^{G}$.
As $\sigma$ permutes the monomials,  we have the following:

\begin{lemma}\label{orbit}
Let $M$ denote the set of monomials of $F[V]$. $F[V]^G$ is spanned
as a vector space by orbit sums of $\rho$-invariant monomials,
i.e. by the set
\[
\{o(m): m\in M^{\rho}\}=\{ m+\sigma (m):\,\, m\in M^{\rho} \} \,\,\cup\,\, \{m:\,\, m\in M^{G}\}.
 \]
\end{lemma}

 Let $f\in F[V]^G_{+}$. We call $f$ expressible if $f$ is in the algebra generated by the invariants whose degrees are strictly smaller than the degree of $f$.

 \begin{lemma}
 Let $m=x_1^{a_1}\dots x_r^{a_r}y_1^{b_1}\dots y_r^{b_r}z_1^{c_1}\dots
 z_s^{c_s}w_1^{d_1}\dots w_s^{d_s}\in M^{\rho}$ such that  $o(m)$ is not expressible. Then $\sum_{1\le j\le s}(c_j+d_j)\le s$.
\end{lemma}

\begin{proof}
Assume by contradiction that $\sum_{1\le j\le s}(c_j+d_j) > s$.
Pick an integer $1\le j\le s$ such that $c_j+d_j\ge 2$. If both
$c_j$ and $d_j$ are non-zero, then $m$ is divisible by the invariant $z_jw_j$.
It follows that $o(m)$ is divisible
by $z_jw_j$, hence $o(m)$ is expressible. Now assume $c_j\ge 2$ and $d_j=0$. Note that $m/z_j\in M^{\rho}$. We
consider the product
\[o(z_j)o(m/z_j)=(z_j+w_j)(m/z_j+\sigma
(m)/w_j)=o(m)+(mw_j/z_j+\sigma(m)z_{j}/w_{j}).\] As $mw_j/z_j$ is
divisible by $z_jw_j$ (because $m$ is divisible by $z_j^2$), the
invariant $f:=mw_j/z_j+\sigma(m)z_{j}/w_{j}$ is divisible by
$z_{j}w_{j}$. Hence $o(m)=o(z_j)o(m/z_j)+f$ is expressible. The
case $c_j=0$ and $d_j\ge 2$ is handled similarly.
\end{proof}

\begin{theorem}\label{GenerationD2p}
$F[V]^G$ is generated by invariants of degree at most $s+\max
\{r,p\}$.
\end{theorem}

\begin{proof}
By Lemma \ref{orbit} it suffices to show that $o(m)$ is
expressible for any monomial $m=x_1^{a_1}\dots
x_r^{a_r}y_1^{b_1}\dots y_r^{b_r}z_1^{c_1}\dots
z_s^{c_s}w_1^{d_1}\dots w_s^{d_s}\in M^{\rho}$ of degree bigger
than or equal to $s+\max \{r,p\}+1$. Also by the previous lemma we
may assume that $\sum_{1\le j\le s}(c_j+d_j)\le s$.  But then
$t:=\sum_{1\le i\le r}(a_i+b_i)\ge \max\{r,p\}+1\ge r+1$, so we
may take $a_1+b_1\ge 2$. As before, not both of $a_1$ and $b_1$
are non-zero because otherwise $o(m)$ is divisible by the
invariant polynomial $x_1y_1$ and so is expressible. So without
loss of generality we assume that $a_1\ge 2,\, b_{1}=0$. Let
$\kappa_F$ denote the character group of $H$, whose elements are
group homomorphisms from $H$ to $F^*$. Note that $\kappa_F \cong
H$. For $1\le i\le r$, let $\kappa_i\in \kappa_F$ denote the
character corresponding to the action  of $H$ on $x_i$. By
construction the character corresponding to the action on $y_i$ is
$-\kappa_i$. Since $\rho (m)=m$ we have $\sum_{1\le i\le
r}(a_i\kappa_i-b_i\kappa_i)=0$. This is an equation in a cyclic
group of order $p$ that contains at least $t\ge p+1$ (not
distinct)  summands. Since $a_1\ge 2$, \cite[Proposition
7.7]{MR1180987} applies and we get non-negative integers $a_i'\le
a_i$ and $b_i'\le b_i$ for $1\le i\le r$ with $0< a_1'< a_1$
satisfying $\sum_{1\le i\le r}(a_i'\kappa_i-b_i'\kappa_i)=0$.
Hence $m_1:=x_1^{a_1'}\dots x_r^{a_r'}y_1^{b_1'}\dots
y_r^{b_r'}z_1^{c_1}\dots z_s^{c_s}w_1^{d_1}\dots w_s^{d_s}$ is
$\rho$-invariant. Thus $m_2:=m/m_1$ is also $\rho$-invariant.
Since $0<a_1'<a_1$, both $m_1$ and $m_2$ are divisible by $x_1$.
Now  consider
\[
(m_1+\sigma (m_1))(m_2+\sigma
(m_2))=o(m)+(m_1\sigma (m_2)+\sigma(m_{1})m_{2}).
\]
As $m_1\sigma (m_2)$ is divisible by $x_1y_1$, so is
$f:=(m_1\sigma (m_2)+\sigma(m_{1})m_{2})$. It follows that
$o(m)=(m_1+\sigma (m_1))(m_2+\sigma(m_2))+f$ is expressible.
\end{proof}
\begin{remark}\label{remark}
 Assume that $V=W_i$ for some $1\le
i\le (p-1)/2$ such that $i$ and $p$ are coprime and set $x=x_1$ and $y=y_1$. Then by the previous
theorem $F[V]^G$ is generated by invariants of degree at most $p$.
But the monomials in $M^{\rho}$ of degree strictly less than $p$
are all divisible by $xy\in M^G$. Furthermore,  the only monomials in
$M^{\rho}$ of degree $p$ are $x^p$ and $y^p$, so it follows from
Lemma \ref{orbit} that $F[V]^G=F[x^p+y^p, xy]$.
\end{remark}

\section{ Separating invariants}

For a finite group $G$ (and a fixed field $F$), let $\beta_{\sep}(G)$ denote
the smallest number $d$ such that for any representation $V$ of $G$ there
exists a separating set of invariants of degree $\le d$.

\begin{theorem}\label{betaSep}
For an algebraically closed field $F$ of characteristic $2$ and $p\ge 3$ odd, we have $\beta_{\sep}(D_{2p})=p+1$.
\end{theorem}

Note that in \cite[Proposition 10 and Example 2]{KohlsKraft}, bounds for
$\beta_{\sep}(D_{2p})$ are given only in characteristics $\ne
2$, and the theorem above was conjectured for $p$ an
odd prime. For example by \cite{KohlsKraft}, when $p$ is an odd prime and equals the
characteristic of $F$, then $\beta_{\sep}(D_{2p^{r}})=2p^{r}$ for any $r\ge 1$.

\begin{proof}
We look at the regular representation $V_{\reg}:=FG$, which
decomposes into
$V_{\reg}=\bigoplus_{i=1}^{\frac{p-1}{2}}W_{i}\oplus\bigoplus_{i=1}^{\frac{p-1}{2}}W_{i}
\oplus W_0$. This can be seen by considering the action of $G$ on
the basis of $FG$ consisting of the elements
$v_{k}:=\sum_{j=0}^{p-1}\lambda^{kj}\rho^{j}$ and
$w_{k}:=\sigma(v_{k})$ for $k=0,\ldots,p-1$, where $\lambda$ is a
primitive $p$th root of unity. Then
$\rho(v_{k})=\lambda^{-k}v_{k}, \,\,
\rho(w_{k})=\sigma\rho^{-1}v_{k}=\lambda^{k}w_{k}$, and $\sigma$
interchanges $v_{k}$ and $w_{k}$. It follows that
$\langle
v_{k},w_{k}\rangle\cong W_k$ if $0\le k\le \frac{p-1}{2}$ and
$\langle
v_{k},w_{k}\rangle\cong W_{p-k}$ if $\frac{p+1}{2}\le k\le p-1$.

By Theorem \ref{GenerationD2p}, $F[V_{\reg}]^{G}$ is generated by
invariants of degree $\le 1+\max\{p,2\frac{p-1}{2}\}=1+p$. Hence
$\beta_{\sep}(G)\le p+1$ by \cite[Corollary 3.11]{MR2414957} (see
also \cite[Proposition 3]{KohlsKraft}). Note that for $p$ a prime,
this follows constructively from Theorem \ref{SepD2pTheorem}. To
prove the reverse inequality, consider $V:=W_{1}\oplus W_{0}$. We
use the notation of section \ref{SetupSection}, so
$F[V]=F[x,y,z,w]$ (omitting indices since $r=s=1$) and look at the
points $v_{1}:=(0,1,1,0)$ and $v_{2}:=(0,1,0,1)$ of $V$. They can
be separated by the invariant $zx^{p}+wy^{p}$. Assume they can be
separated by an invariant of degree less or equal than $p$. By Lemma
\ref{orbit}, $F[V]^{G}$ is generated by invariant monomials $m\in
F[V]^{G}$ and orbit sums $m+\sigma(m)$ of $\rho$-invariant
monomials $m\in F[V]^{\rho}$. If such an element separates $v_{1}$
and $v_{2}$, we have $m(v_{1})\ne m(v_{2})$ or $(m+\sigma
m)(v_{1})\ne (m+\sigma m)(v_{2})$ respectively. The latter implies
$m(v_{1})\ne m(v_{2})$ or $\sigma(m)(v_{1})\ne \sigma(m)(v_{2})$.
Replacing $m$ by $\sigma(m)$ if necessary, we thus have a
$\rho$-invariant monomial $m$ separating $v_{1},v_{2}$ of degree
$\le p$. Therefore, $x$ does not appear in $m$, so
$m=y^{a}z^{b}w^{c}$. First assume $a=0$. If $b=c$, then $m$ is
$G$-invariant, and does not separate $v_{1},v_{2}$. If $b\ne c$,
then $m$ is not $G$-invariant, and
$m+\sigma(m)=z^{b}w^{c}+z^{c}w^{b}$ does not separate
$v_{1},v_{2}$. So $a>0$. As $m$ is $\rho$-invariant, we have $a\ge
p$. Since $\deg m\le p$, we have $a=p$ and $b=c=0$. Then
$m+\sigma(m)=y^{p}+x^{p}$ does not separate $v_{1},v_{2}$. We have
a contradiction.
\end{proof}

Theorem \ref{betaSep} gives an upper bound for the degrees of polynomials in a
separating set. In the following, under the additional assumption
that $p>2$ is a prime, we  construct a separating set explicitly.
We use again the notation of section \ref{SetupSection}. We assume
that $V$ is a faithful $G$-module. In particular we have $r\ge 1$. Let $1\le i\le r-1$ be arbitrary. Since the action of
$\rho$ is non-trivial on each of the variables $x_r,y_1, \dots ,
y_{r-1}$ there exists a positive integer $n_i\le p-1$ such that
$x_ry_i^{n_i}$ and $x_rx_i^{p-n_i}$ are invariant under the action
of $\rho$. We thus get invariants
\[f_{i}:=x_ry_i^{n_i}+y_rx_i^{n_i},\quad  g_{i}:=x_rx_i^{p-n_i}+y_ry_i^{p-n_i}\in F[V]^G \quad \text{ for }i=1,\ldots,r-1.\]  For $1\le i\le r-1$
and $1\le j\le s$ we also define
\[f_{i,j}:=x_ry_i^{n_i}z_j+y_rx_i^{n_i}w_j,   \quad h_{j}:=x_{r}^{p}z_{j}+y_{r}^{p}w_{j} \in F[V]^{G}.\]
Set
$V'=\bigoplus_{i=1}^{r-1}W_{m_i}\oplus\bigoplus_{i=1}^sW_0$.

\begin{theorem}\label{SepD2pTheorem}
Let $p>2$ be a prime. Let $S$ be a separating set for $V'$. Then $S$ together with the
set $$T=\{x_ry_r,x_r^p+y_r^p, f_i, g_i, f_{i,j}, h_{j} \mid \; 1\le i\le
r-1,\;  1\le j\le s\}$$ of invariant polynomials  is a separating set for $V$.
\end{theorem}

Note that a separating set for $\bigoplus_{i=1}^s W_0$ is given in \cite{MR14}.

\begin{proof}
We have a  surjection $V\rightarrow V':(a_1, \dots , a_r, b_1,
\dots , b_r, c_1,\dots ,c_s, d_1, \dots, d_s)\rightarrow (a_1,
\dots , a_{r-1}, b_1, \dots , b_{r-1}, c_1,\dots ,c_s, d_1, \dots
,d_s)$ which is $G$-equivariant. Therefore by \cite[Theorem
1]{kohlssezer} it suffices to show that the polynomials in $T$
separate any pair of vectors $v_1$ and $v_2$ in different
$G$-orbits that agree everywhere except $r$-th and $2r$-th
coordinates. So we take $v_1=(a_1, \dots , a_r, b_1, \dots , b_r,
c_1,\dots ,c_s, d_1, \dots, d_s)$ and $v_2=(a_1, \dots ,
a_{r-1},a_r', b_1, \dots ,b_{r-1}, b'_r, c_1,\dots ,c_s, d_1,
\dots, d_s)$. Assume by way of contradiction that no polynomial in
$T$ separates $v_1$ and $v_2$. Since
$\{x_ry_r,x_r^p+y_r^p\}\subseteq T$ is a separating set for
$W_{m_r}$ by Remark \ref{remark}, we may further take   that
$(a_r,b_r)$ and $(a_r',b_r')$ are in the same $G$-orbit. Consequently,  there are
two cases.

First we assume that there exists an integer $t$ such
that $(a_r',b_r')=\rho^t (a_r,b_r)$. Hence
$a_r'=\lambda^{-tm_r}a_r$ and $b_r'=\lambda^{tm_r}b_r$. Set
$c:=\lambda^{-tm_r}$. Notice that $a_r$ and $b_r$ can not be zero
simultaneously because otherwise $v_1=v_2$. Without loss of
generality we take $a_r\neq 0$. Also if $a_i=b_i=0$ for all $1\le
i\le r-1$ then we have $\rho^t (v_1)=v_2$, hence $r>1$ and there is  an index
$1\le q \le r-1$ such that at least one of  $a_q$ or $b_q$ is
non-zero. We show in fact both $a_q$ and $b_q$  are non-zero together with
$b_r$.  First assume that $a_q\neq 0$. If one of $b_q$ or $b_r$ is
zero, then $g_q(v_1)=a_ra_q^{p-n_{q}}$ and
$g_q(v_2)=ca_ra_q^{p-n_{q}}$. This yields a contradiction because
$g_q(v_1)=g_q(v_2)$. Next assume that  $b_q\neq 0$. If one of
$a_q$ or $b_r$  is zero then $f_q(v_1)=a_rb_q^{n_q}$ and
$f_q(v_2)=ca_rb_q^{n_q}$, yielding a contradiction again. In fact,
applying the same argument using the invariant $g_i$ (or $f_i$)
shows that for $1\le i\le r-1$ we have: $a_i\neq 0$ if and only if
$b_i\neq 0 $. We claim that   $a_i^p=b_i^p$  for $1\le i\le r-1$. Clearly
we may assume $a_i\neq 0$. From $f_i(v_1)=f_i(v_2)$ we get
$(1+c)a_rb_i^{n_i}=(1+c^{-1})b_ra_i^{n_i}$. Similarly from
$g_i(v_1)=g_i(v_2)$ we have
$(1+c)a_ra_i^{p-n_i}=(1+c^{-1})b_rb_i^{p-n_i}$. It follows that
$$c^{-1}=\frac{a_rb_i^{n_i}}{b_ra_i^{n_i}}=\frac{a_ra_i^{p-n_i}}{b_rb_i^{p-n_i}}.$$
This establishes the claim.  For $1\le i\le r-1$, let $e_i$ denote the smallest non-negative integer such that
$b_i=\lambda^{e_i} a_i$. We also have $b_r=c\lambda^{e_in_i}a_r$ provided $a_i\neq 0$.
 We now show that $c_j=d_j$ for all $1\le j\le s$. From
$f_{q,j}(v_1)=f_{q,j}(v_2)$ we have
$c_ja_rb_q^{n_q}+d_jb_ra_q^{n_q}=cc_ja_rb_q^{n_q}+c^{-1}d_jb_ra_q^{n_q}$.
Putting $b_q=\lambda^{e_q}a_q$ and $b_r=c\lambda^{e_qn_q}a_r$ we
get
$c_ja_r\lambda^{e_qn_q}a_q^{n_q}+d_jc\lambda^{e_qn_q}a_ra_q^{n_q}=cc_ja_r\lambda^{e_qn_q}a_q^{n_q}+c^{-1}d_jca_r\lambda^{e_qn_q}a_q^{n_q}$
which gives $c_j+cd_j=cc_j+d_j$. This implies $c_j=d_j$ as desired
because $1+c\neq 0$ . We now have $v_1=(a_1, \dots , a_r,
\lambda^{e_1}a_1, \dots , \lambda^{e_{r-1}}a_{r-1},
c\lambda^{e_qn_q}a_r, c_1,\dots ,c_s, c_1, \dots, c_s)$ and
$v_2=(a_1, \dots , a_{r-1},ca_r, \lambda^{e_1}a_1, \dots ,
\lambda^{e_{r-1}}a_{r-1}, \lambda^{e_qn_q}a_r, c_1,\dots ,c_s,
c_1, \dots, c_s)$. Since $0< m_r<p$, there exists an integer $0\le
h\le p-1$ such that $-hm_r+e_qn_q\equiv 0 \; \mod p$. We obtain a
contradiction by showing that $\rho^h \sigma (v_1)=v_2$. Since the
action of $\rho$ on the last $2s$ coordinates is trivial it
suffices to show that $\lambda^{-hm_i}b_i=a_i$ for $1\le i\le r-1$ and $\lambda^{-hm_r}b_r=ca_r$.
Hence we need to show $-hm_i+e_i\equiv 0 \; \mod p$ for $1\le i\le
r-1$ when $a_{i}\ne 0$, and $-hm_r+e_qn_q \equiv 0 \; \mod p$. The second equality
follows by the choice of $h$. So assume that $1\le i \le r-1$ and
$a_i\neq 0$. We have $m_r-n_im_i\equiv 0 \; \mod p$ because
$x_ry_i^{n_i}$ is invariant under the action of $\rho$. It follows
that $e_qn_q-hn_im_i\equiv 0 \; \mod p$. But since $e_in_i\equiv
e_qn_q$ (as $b_{r}=c\lambda^{e_{i}n_{i}}a_{r}=c\lambda^{e_{q}n_{q}}a_{r}$) we have $n_i(e_i-hm_i)\equiv 0 \; \mod p$. Since $n_i$ is
non-zero modulo $p$ we have $e_i-hm_i\equiv 0 \; \mod p$ as
desired.

Next we consider the case  $(a_r',b_r')=\rho^t \sigma (a_r,b_r)$
for some integer $t$. Hence $a_r'=\lambda^{-tm_r}b_r$ and
$b_r'=\lambda^{tm_r}a_r$. Set $c:=\lambda^{-tm_r}$. As in the first
case one of $a_r$ or $b_r$ is non-zero, so without loss of generality we take
$a_r\neq 0$.
As $h_{j}(v_{1})=h_{j}(v_{2})$ for $1\le j\le s$, we get
$(a_{r}^{p}+a_{r}'^{p})c_{j}=(b_{r}^{p}+b_{r}'^{p})d_{j}$, which implies
$(a_{r}^{p}+b_{r}^{p})c_{j}=(a_{r}^{p}+b_{r}^{p})d_{j}$. If
$a_{r}^{p}=b_{r}^{p}$, we have $b_{r}=\lambda^{l}a_{r}$ for some
$l$. Then we have
$(a_{r}',b_{r}')=(\lambda^{-tm_{r}+l}a_{r},\lambda^{tm_r-l}b_{r})\in \langle
\rho \rangle \cdot (a_{r},b_{r})$, so we are again in the first case.
Therefore we can assume $a_{r}^{p}\ne b_{r}^{p}$, and we get
$c_{j}=d_{j}$ for all $1\le j\le s$.
Now, if $a_i=b_i=0$ for all $1\le i\le r-1$, then $v_{2}=\rho^{t}\sigma(v_{1})$. Hence $r>1$ and there is  an index $1\le q \le
r-1$ such that at least one of  $a_q$ or $b_q$ is non-zero. Let
$1\le i\le r-1$. From $f_i(v_1)=f_i(v_2)$ we get
$a_rb_i^{n_i}+b_ra_i^{n_i}=cb_rb_i^{n_i}+c^{-1}a_ra_i^{n_i}$ and
so $a_i^{n_i}(c^{-1}a_r+b_r)=b_i^{n_i}(a_r+cb_r)$. Note that
$c^{-1}a_r+b_r\neq 0$ because otherwise $v_1=v_2$. So we have
$a_i^{n_i}=cb_i^{n_i}$. Along the same lines, from
$g_i(v_1)=g_i(v_2)$ we obtain $b_i^{p-n_i}=ca_i^{p-n_i}$. It
follows that $a_i^p=b_i^p$. As before, for $1\le i\le r-1$ let  $e_i$ denote
the smallest non-negative integer such that $b_i=\lambda^{e_i}a_i$. We also
have $c=\lambda^{-n_ie_i}$ for all $1\le i\le r-1$ with  $a_i\neq 0$.
We have $v_1=(a_1, \dots , a_r, \lambda^{e_1}a_1,
\dots , \lambda^{e_{r-1}}a_{r-1}, b_r, c_1,\dots
,c_s, c_1, \dots, c_s)$ and
 $v_2=(a_1, \dots a_{r-1}, cb_r, \lambda^{e_1}a_1,
\dots , \lambda^{e_{r-1}}a_{r-1}, c^{-1}a_r, c_1,\dots
,c_s, c_1, \dots, c_s)$. We finish the proof by demonstrating that $v_1$ and $v_2$ are in the same orbit.
Since
$0< m_r<p$, there exists an integer $0\le h\le p-1$ such that $\lambda^{-hm_r}=c$. Equivalently, $-hm_r+e_qn_q\equiv 0 \; \mod p$. We claim that $\rho^h \sigma (v_1)=v_2$. Since $c_j=d_j$ for $1\le j\le s$ and the action of $\rho$ on the last $2s$ coordinates is trivial we just need to show that
$\lambda^{-hm_i}b_i=a_i$ for $1\le i\le r-1$ and $\lambda^{-hm_r}b_r=cb_{r}$.
 Since the last equation is taken care of by construction we just need to show
 $-hm_i+e_i \equiv 0 \; \mod p$ for $1\le i\le r-1$ when $a_{i}\ne 0$. We get
 $e_{i}n_{i}\equiv e_{q}n_{q}$ from
 $c=\lambda^{-e_{i}n_{i}}=\lambda^{-e_{q}n_{q}}$. Now the proof can be
 finished by exactly the same argument as in the first case.
\end{proof}

\bibliographystyle{plain}
\bibliography{Bibliography_Version_10}
 \end{document}